\documentclass{amsart}
\usepackage{amsfonts}
\usepackage{graphicx}
\usepackage{amscd}
\usepackage{amsmath}
\usepackage{amssymb}
 \theoremstyle{plain}

\newtheorem{conjecture}{Conjecture}
\newtheorem{corollary}{Corollary}

\newtheorem{lemma}{Lemma}

\newtheorem{proposition}{Proposition}

\newtheorem*{thmElsfi}{Theorem 2.2 [El Soufi]}
\newtheorem*{thmPerdomo}{Theorem 1.3 [Perdomo]}

\newtheorem*{maintheorem}{Theorem 1}
\numberwithin{equation}{section}

\begin{document}

\title{On the Index of Constant Mean Curvature Hypersurfaces}
\author{E. Colberg, A.M. de Jesus, K. Kinneberg, G. Silva Neto}

\begin{abstract}
In 1968, Simons \cite{Simons68} introduced the concept of index for hypersurfaces immersed into the Euclidean sphere $\mathbb{S}^{n+1}$. Intuitively, the index measures the number of independent directions in which a given hypersurface fails to minimize area. The earliest results regarding the index focused on the case of minimal hypersurfaces. Many such results established lower bounds for the index. More recently, however, mathematicians have generalized these results to hypersurfaces with constant mean curvature. In this paper, we consider hypersurfaces of constant mean curvature immersed into the sphere and give lower bounds for the index under new assumptions about the immersed manifold.
\end{abstract}

\date{\today}
\thanks{Supported by NSF-OISE: 0526008 and CNPq: 451778/2008-1}

\maketitle

\section{Introduction}

The index of minimal hypersurfaces immersed into the Euclidean sphere $\mathbb{S}^{n+1}$ appeared in
the seminal work of Simons \cite{Simons68}. Intuitively, the index measures the number of independent
directions in which the hypersurface fails to minimize area. In this paper, Simons proved, among many results,
that the index of such hypersurfaces is greater than or equal to one, with equality only at totally geodesic spheres.
Later, Urbano \cite{Urbano90} established the lower bound for the case of minimal surfaces into the Euclidean sphere $\mathbb{S}^{3}$. Indeed, he proved that the lower bound is attained on the Clifford torus. Following Urbano's result, El Soufi \cite{Soufi93} proved that the index of a compact non-totally geodesic minimal hypersurface of the sphere not only must be greater than one, but in fact, must be greater than or equal to $n+3$ which correspond to the index of a minimal Clifford torus. Therefore, we have the following result.
\begin{thmElsfi}
Let $\Sigma^n$ be a compact orientable minimal hypersurface immersed into $\mathbb{S}^{n+1}$. Then
\begin{enumerate}
\item either ${Ind}(\Sigma) = 1$ (and $\Sigma$ is a totally geodesic equator
$\mathbb{S}^n \subset \mathbb{S}^{n+1}$);
\item or ${Ind}(\Sigma) \geq n+3$.
\end{enumerate}
\end{thmElsfi}

On the other hand, apart from totally geodesic equators of the sphere, the easiest minimal hypersurfaces
in $\mathbb{S}^{n+1}$ are the minimal Clifford tori. For that reason, it has been conjectured that the minimal Clifford tori are the only minimal hypersurfaces of the sphere with index $n+3$. Thus, we have the following.
\begin{conjecture}
Let $\Sigma^n$ be a compact orientable minimal hypersurface immersed into $\mathbb{S}^{n+1}$. Then
\begin{enumerate}
\item either ${Ind}(\Sigma) = 1$ (and $\Sigma$ is a totally geodesic equator
$\mathbb{S}^n \subset \mathbb{S}^{n+1}$);
\item or ${Ind}(\Sigma) \geq n+3$, with equality if and only if
$\Sigma$ is a minimal Clifford torus
$\mathbb{S}^k(\sqrt{\frac{k}{n}}) \times \mathbb{S}^{n-k}(\sqrt{\frac{n-k}{n}})$.
\end{enumerate}
\end{conjecture}
This conjecture has not yet been proven, but many mathematicians
have solved some particular cases. For instance, as pointed out Urbano
\cite{Urbano90} proved that this conjecture is true when $n=2$.
Later on, Brasil, Delgado and Guadalupe, \cite{GBD99} showed that
the conjecture is true for every dimension $n$ provided the scalar
curvature is constant. But, in fact, this follows from a early result due
to Nomizu and Smyth \cite{NS68}.
On the other hand, Perdomo \cite{Perdomo01} has also showed that the conjecture is true for every dimension $n$ with a
hypothesis about the symmetry of the hypersurface. More exactly, he proved.
\begin{thmPerdomo}
Let $\Sigma^n$ be a compact orientable minimal hypersurface immersed
into $\mathbb{S}^{n+1}$, invariant under the antipodal map, and not
a totally geodesic equator. Then ${Ind}(\Sigma) \geq n+3$, with
equality if and only if $\Sigma$ is a minimal Clifford torus.
\end{thmPerdomo}
In his paper, Perdomo observed that the symmetry assumption seems weak
because all of the many known embedded minimal hypersurfaces in  dimension bigger than $3$ of the sphere do have antipodal symmetry.

Besides minimal hypersurfaces of $\mathbb{S}^{n+1}$, a natural generalization is the case of
hypersurfaces with constant mean curvature (CMC) hypersurfaces. In this case, it is more natural
geometrically  to  study  the  weak  index  of  such  hypersurfaces.

Barbosa, do Carmo and Eschenburg \cite{BdCE88} proved that totally umbilical round spheres $\mathbb{S}^n(r) \subset \mathbb{S}^{n+1}$ are the only compact weakly stable CMC hypersurfaces in $\mathbb{S}^{n+1}$. It has also been proven that the minimal Clifford tori have weak index equal to $n+2$ when regarded as CMC hypersurfaces. These results lead us to a natural conjecture.
\begin{conjecture}
Let $\Sigma^n$ be a compact orientable CMC hypersurface immersed into $\mathbb{S}^{n+1}$. Then
\begin{enumerate}
\item either ${Ind}_T(\Sigma) = 0$ (and $\Sigma$ is a totally umbilical sphere in $\mathbb{S}^{n+1}$);
\item or ${Ind}_T(\Sigma) \geq n+2$, with equality if and only if $\Sigma$ is a CMC Clifford torus $\mathbb{S}^k(r) \times \mathbb{S}^{n-k}(\sqrt{1-r^2})$ with radius $\sqrt{\frac{k}{n+2}} \leq r \leq \sqrt{\frac{k+2}{n+2}}$.
\end{enumerate}
\end{conjecture}
As was first done on the minimal case, recently, Al\'ias, Brasil Jr. and Perdomo \cite{ABP07} showed that this conjecture is true under the additional hypothesis that the hypersurface has constant scalar curvature.

In this paper, we shall present a lower bound for the index of a CMC hypersurface under an additional hypothesis on the second fundamental form. More exactly, we obtain the following result.
\begin{maintheorem}
Let $x:\Sigma^n\hookrightarrow\mathbb{S}^{n+1}$ be a CMC isometric immersion of a compact oriented
manifold $\Sigma^n$. Assume that $|A|^2-2nH^2\geq0$.
\begin{enumerate}
\item If $H=\pm1$ then ${Ind}_T(\Sigma^n)\geq n+2.$
\item If $\int_{\Sigma}|\nabla l_v|^2 d\Sigma \geq n\int_{\Sigma}l_{v}^{2}d\Sigma$ and
$|H|\leq 1,$ then ${Ind}_T (\Sigma^n)\geq n+2.$
\item If $\int_{\Sigma}|\nabla l_v|^2 d\Sigma \leq n\int_{\Sigma}l_{v}^{2}d\Sigma$
and $|H|\geq 1,$ then ${Ind}_T (\Sigma^n)\geq n+2.$
\end{enumerate}
\end{maintheorem}

\section{Background}

Let us consider $\varphi: \Sigma^n \hookrightarrow \mathbb{S}^{n+1}$, a compact orientable hypersurface immersed
into the unit Euclidean sphere $\mathbb{S}^{n+1}$. We will denote by $A$ the shape operator of $\Sigma$
with respect to a globally defined normal unit vector field $N$.
That is, $A: \mathcal{X}(\Sigma) \rightarrow \mathcal{X}(\Sigma)$ is the endomorphism defined by
\begin{equation} \label{1}
AX = -\tilde{\nabla}_X N = -\bar{\nabla}_X N, {     } X \in
\mathcal{X}(\Sigma)
\end{equation}
where $\tilde{\nabla}$ and $\bar{\nabla}$ denote, respectively, the
Levi-Civita connection on $\mathbb{R}^{n+2}$ and $\mathbb{S}^{n+1}$.
The mean curvature of $\Sigma$ is defined as $$H = \frac{1}{n}
{tr}(A).$$

If $\nabla$ denotes the Levi-Civita connection on $\Sigma$, then the Gauss formula for the
immersion $\varphi$ is given by
\begin{equation} \label{2}
\tilde{\nabla}_X Y = \bar{\nabla}_X Y - \langle X,Y \rangle \varphi =
\nabla_X Y + \langle AX,Y \rangle N - \langle X,Y \rangle \varphi
\end{equation}
for every tangent vector field $X,Y \in \mathcal{X}(\Sigma)$. The covariant derivative
of $A$ is defined by
$$\nabla A(X,Y) = (\nabla_Y A)X = \nabla_Y (AX) - A(\nabla_Y X), {     } X,Y \in \mathcal{X}(\Sigma),$$
and the Codazzi equation is given by
\begin{equation} \label{3}
\nabla A(X,Y) = \nabla A(Y,X)
\end{equation}
for any $X,Y \in \mathcal{X}(\Sigma)$.

Every smooth function $f \in C^{\infty}(\Sigma)$ induces a normal variation
$\varphi_t : \Sigma \rightarrow \mathbb{S}^{n+1}$ of the original immersion $\varphi$,
given by $$\varphi_t = \cos(tf(p))\varphi(p) + \sin(tf(p))N(p) .$$ Since $\varphi_0 = \varphi$ is an
immersion and this is an open condition, there exists an $\varepsilon > 0$ such that every $\varphi_t$ is
also an immersion, for $|t| < \varepsilon$. Then we can consider the area
function $\mathcal{A}: (-\varepsilon, \varepsilon) \rightarrow \mathbb{R}$ given by
$$\mathcal{A}(t) = {Area}(\Sigma_t) = \int_{\Sigma}{d\Sigma_t},$$ where $\Sigma_t$
stands for the manifold $\Sigma$ endowed with the metric induced by $\varphi_t$ from the
Euclidean metric on $\mathbb{S}^{n+1}$, and $d\Sigma_t$ is the $n$-dimensional area element
of that metric on $\Sigma$. The first variation formula for the area (\cite{Lawson80}, Chapter 1, Theorem 4)
establishes that
\begin{equation} \label{4}
\mathcal{A}'(0) = \frac{d}{dt}\left[\mathcal{A}(t)\right]_{t=0} = -n \int_{\Sigma}{fHd\Sigma}.
\end{equation}

The stability operator of this variation problem is given by the second variation formula for the
area (\cite{Lawson80}, Chapter 1, Theorem 32),
\begin{equation} \label{5}
\mathcal{A}''(0) = \frac{d^2}{dt^2}\left[\mathcal{A}(t)\right]_{t=0}
= -\int_{\Sigma}{f(\Delta f + |A|^2f + nf)d\Sigma} = -\int_{\Sigma}{fJfd\Sigma}.
\end{equation}
Here, $J = \Delta + |A|^2 + n$ is the Jacobi operator, where
$\Delta$ stands for the Laplacian operator of $\Sigma$ and $|A|^2 =
{tr}(A^2)$.

When working with constant mean curvature hypersurfaces, it is often convenient to use the traceless
second fundamental form given by $\phi = A - HI$, where $I$ denotes the identity operator on $\mathcal{X}(\Sigma)$.

As a consequence of the first variation formula for the area \ref{4}, we have that $\Sigma$ has constant
mean curvature (not necessarily zero) if and only if $A'(0) = 0$ for every smooth function
$f \in C^{\infty}(\Sigma)$ satisfying the additional condition $\int_{\Sigma}{f d\Sigma} = 0$.

The Jacobi operator induces the quadratic form $Q: C^{\infty}(\Sigma) \rightarrow \mathbb{R}$ acting
on the space of smooth functions on $\Sigma$ defined by $$Q(f) = -\int_{\Sigma}{fJfd\Sigma}.$$

There are two different notions of stability and index, the strong
stability and strong index, denoted by ${Ind}(\Sigma)$, and the weak
stability and weak index, denoted by ${Ind}_T(\Sigma)$. Thus, the
strong index is
$${Ind}(\Sigma) = \max \{ \dim V : V \subset C^{\infty}(\Sigma), \, Q(f) < 0, \, \forall\, f \in V \}$$
and $\Sigma$ is called strongly stable if and only if ${Ind}(\Sigma)
= 0$. On the other hand, the weak index is
$${Ind}_T(\Sigma) = \max \{ \dim V : V \subset C^{\infty}_T(\Sigma),\, Q(f) < 0,\,\forall\, f \in V \}, $$
where $$ C^{\infty}_T(\Sigma) = \{ f \in C^{\infty}(\Sigma) :
\int_{\Sigma}{fd\Sigma} = 0 \} , $$ and $\Sigma$ is called weakly
stable if and only if ${Ind}_T{\Sigma} = 0$.

\subsection{Preliminary Calculations}

Let $\varphi: \Sigma^n \hookrightarrow \mathbb{S}^{n+1}$ be a constant mean curvature isometric immersion of a
compact oriented manifold $\Sigma^n$. For a fixed arbitrary vector $v \in \mathbb{R}^{n+2}$, we will
consider the support functions $l_v = \langle \varphi, v \rangle$ and $f_v = \langle N,v \rangle$ defined on $\Sigma$.
A standard computation, using equations \ref{1} and \ref{2}, shows that the gradient and the
Hessian of the functions $l_v$ and $f_v$ are given by
\begin{equation} \label{6}
\nabla l_v = v^T
\end{equation}
\begin{equation} \label{7}
\nabla^2 l_v(X,Y) := \langle \nabla_X \nabla l_v , Y \rangle = -l_v \langle X,Y \rangle + f_v \langle AX,Y \rangle
\end{equation}
and
\begin{equation} \label{8}
\nabla f_v = -A(v^T)
\end{equation}
\begin{equation} \label{9}
\nabla^2 f_v(X,Y):= \langle \nabla_X \nabla f_v , Y \rangle = -
\langle \nabla A(v^T,X),Y \rangle + l_v \langle AX, Y \rangle - f_v
\langle AX,AY \rangle,
\end{equation}
for every tangent vector field $X,Y \in \mathcal{X}(\Sigma)$. Here, $v^T = v - l_v \varphi - f_v N \in \mathcal{X}(\Sigma)$
denotes the tangential component of $v$ along the immersion $\varphi$.

Equation \ref{7} directly yields $$\Delta l_v = {tr}(\nabla^2 l_v) =
-nl_v + nHf_v.$$ Using the Codazzi equation \ref{3} in equation
\ref{9} we also obtain
$$\Delta f_v = -n \langle v^T , \nabla H \rangle + nH l_v - |A|^2 f_v.$$
In particular , if $H$ is constant we have  $$\Delta f_v =nH l_v -
|A|^2 f_v.$$From here, a direct calculation yields
$$Jl_v = |A|^2 l_v + nHf_v$$ and $$Jf_v = nHl_v + nf_v.$$

\section{Results}

Consider the function $\psi_v = l_v - Hf_v$. Since $|\phi|^2 = |A|^2
- nH^2$, we obtain $$J \psi_v = |\phi|^2 l_v $$ and
$$\int_{\Sigma}{\psi_v d\Sigma} = 0 . $$ Then a straightforward computation yields

\begin{equation} \label{23}
Q(\psi_v) = -\int_{\Sigma}{|\phi|^2 l_v^2 d\Sigma} +
\int_{\Sigma}{|\phi|^2 Hl_vf_v d\Sigma}.
\end{equation}

\begin{proposition}
Let $\varphi: \Sigma^n \hookrightarrow \mathbb{S}^{n+1}$ be a constant mean curvature isometric
immersion of a compact oriented manifold $\Sigma^n$. If $e_1, \ldots, e_{n+2}$ stand for the canonical
vector field of $\mathbb{R}^{n+2}$, then
$$ \sum_{i=1}^{n+2}Q(\psi_{e_i}) = -\int_{\Sigma}|\phi|^2 d\Sigma. $$
\end{proposition}

\begin{proof}First, notice that equation \ref{23} yields

$$\sum_{i=1}^{n+2}Q(\psi_{e_i})=-\int_{\Sigma}|\phi | ^2\sum_{i=1}^{n+2}l_{e_i}^2d\sigma
+\int_{\Sigma}H| \phi |^2\sum_{i=1}^{n+2}l_{e_i}f_{e_i}d\sigma.$$

Taking into account that
$$\langle x,x \rangle=\left \langle \sum_{i=1}^{n+2}
\langle x,e_i \rangle e_i,\sum_{j=1}^{n+2}\langle x,e_j\rangle
e_j\right \rangle =\sum_{i=1}^{n+2}\langle
x,e_i\rangle^2=\sum_{i=1}^{n+2}l_{e_i}^2$$ and
$$\langle x,N\rangle=\left\langle\sum_{i=1}^{n+2}\langle x,e_i\rangle e_i,\sum_{j=1}^{n+2}
\langle N,e_j \rangle e_j\right\rangle=\sum_{i=1}^{n+2}\langle
x,e_i\rangle \langle N,e_i\rangle =\sum_{i=1}^{n+2}l_{e_i}f_{e_i}$$
we arrive at
$$\sum_{i=1}^{n+2}Q(\psi_{e_i})=-\int_{\sigma}| \phi
|^{2}d\sigma,$$ which finishes the proof of proposition.
\end{proof}

\begin{corollary}
Let $\varphi: \Sigma^n \hookrightarrow \mathbb{S}^{n+1}$ be a
constant mean curvature isometric immersion of a compact oriented
manifold $\Sigma^n$. Then, up to totally umbilical spheres,
${Ind}_T(\Sigma) \geq 1$.
\end{corollary}

\begin{proof}
If $\Sigma^n$ is not totally umbilical, then $|\phi|^2 > 0$. Thus,
$$ \sum_{i=1}^{n+2}{Q(\psi_{e_i})} < 0 $$ so there is at least one
$\psi_{e_i}$ such that $Q(\psi_{e_i}) < 0$. Hence, ${Ind}_T(\Sigma)
\geq 1$.
\end{proof}

\subsection{Main Result}

Our main result is the following theorem.

\begin{maintheorem}
Let $\varphi: \Sigma^n \hookrightarrow \mathbb{S}^{n+1}$ be a constant mean curvature isometric immersion
of a compact oriented manifold $\Sigma^n$. Assume that $|A|^2 - 2nH^2 \geq 0$ and $\Sigma^n$ is not totally umbilical.
Then the following are true.
\begin{itemize}
\item If $H = \pm 1$, then ${Ind}_T(\Sigma^n) \geq n+2$.
\item If $\int_{\Sigma}{|\nabla l_v|^2 d\Sigma} \geq n\int_{\Sigma}{l_v^2 d\Sigma}$ for all $v \in \mathbb{R}^{n+2}$
and $|H| \leq 1$, then ${Ind}_T(\Sigma^n) \geq n+2$.
\item If $\int_{\Sigma}{|\nabla l_v|^2 d\Sigma} \leq n\int_{\Sigma}{l_v^2 d\Sigma}$ for all $v \in \mathbb{R}^{n+2}$
and $|H| \geq 1$, then ${Ind}_T(\Sigma^n) \geq n+2$.
\end{itemize}
\end{maintheorem}

Our first objective is to show that if $W = {span} \{ \psi_{e_i}
\}_{i=1}^{n+2}$, where $\{ e_i \}_{i=1}^{n+2}$ is the canonical
frame of $\mathbb{R}^{n+2}$, then (up to totally umbilical spheres)
$\dim W = n+2$. We will then use this fact to prove our main result.

\begin{lemma} \label{dimW}
Let $\varphi: \Sigma^n \hookrightarrow \mathbb{S}^{n+1}$ be a constant mean curvature isometric immersion of a
compact oriented manifold $\Sigma^n$. Then, up to totally umbilical spheres, $\dim W = n+2$.
\end{lemma}

\begin{proof}
Let us suppose that $\{ \psi_{e_1}, \ldots, \psi_{e_{n+2}} \}$ is a
dependent set, where $\{ e_1, \ldots, e_{n+2} \}$ is the canonical
frame of $\mathbb{R}^{n+2}$. Then, there exist non-null real
constants $a_1, \ldots, a_{n+2}$ such that $$ \sum_{i=1}^{n+2}a_i
\psi_{e_i} = 0 . $$ Thus, considering $v = \sum_{i=1}^{n+2}a_i e_i$,
we conclude that $l_v = Hf_v$. We then know that $0 = \Delta(l_v -
Hf_v) = -nl_v + nHf_v + H|A|^2f_v - nH^2l_v = |\phi|^2 l_v$. We now
claim that this yields $|\phi|^2 = 0$. Indeed, if there exists some
$p \in \Sigma$ such that $|\phi|^2(p) \neq 0$, then there is a
neighborhood $\mathcal{U}$ of $p$ such that $|\phi|^2(q) \neq 0$ for
all $q \in \mathcal{U}$. Thus, $l_v(q) = 0$ for all $q \in
\mathcal{U}$. Hence, $\varphi(\Sigma^n) = \mathbb{S}^{n+1}$ on
$\mathcal{U}$. Since $H$ is constant, we have $\varphi(\Sigma^n) =
\mathbb{S}^{n}$, so $\varphi(\Sigma)$ is totally umbilical.
Therefore, we have $|\phi|^2 \equiv 0$. But this also implies that
$\varphi(\Sigma)$ is totally umbilical. Hence, up to totally
umbilical spheres, $\{ \psi_{e_1}, \ldots, \psi_{e_{n+2}} \}$ is an
independent set.
\end{proof}

We now prove the main theorem.

\begin{proof}
By Lemma \ref{dimW}, it suffices to show that $Q(\psi_v) < 0$ for all $v \in \mathbb{R}^{n+2}$. We observe
that $l_v \Delta f_v = -|A|^2 l_v f_v + nH l_v^2$, so
$$ \int_{\Sigma}H |A|^2 l_v f_v d\Sigma = -H \int_{\Sigma}l_v \Delta f_v d\Sigma + nH^2 \int_{\Sigma}l_v^2 d\Sigma. $$
Also, by Green's formula and the compactness of $\Sigma^n$,
$\int_{\Sigma}l_v \Delta f_v d\Sigma = \int_{\Sigma}f_v \Delta l_v
d\Sigma$, so $$ \int_{\Sigma}H |A|^2 l_v f_v d\Sigma = -H
\int_{\Sigma}f_v \Delta l_v d\Sigma + nH^2 \int_{\Sigma}l_v^2
d\Sigma . $$ Therefore, we have the following.
\begin{eqnarray*}
Q(\psi_v) &= &-\int_{\Sigma}|\phi|^2 l_v^2 d\Sigma + \int_{\Sigma}|\phi|^2 Hl_vf_v d\Sigma \\
&=& -\int_{\Sigma}|\phi|^2 l_v^2 d\Sigma + \int_{\Sigma}H |A|^2 l_v
f_v d\Sigma -n \int_{\Sigma}H^3 l_v f_v d\Sigma \\
&= &-\int_{\Sigma}|\phi|^2 l_v^2 d\Sigma -H \int_{\Sigma}f_v \Delta
l_v d\Sigma + nH^2 \int_{\Sigma}l_v^2 d\Sigma
 - nH^3 \int_{\Sigma}l_v f_v d\Sigma \\
&= &-\int_{\Sigma}|\phi|^2 l_v^2 d\Sigma + nH(1-H^2)
\int_{\Sigma}l_v f_v d\Sigma +
nH^2 \int_{\Sigma}(l_v^2 - f_v^2) d\Sigma \\
&= &-\int_{\Sigma}(|A|^2 - 2n)l_v^2 d\Sigma - n \int_{\Sigma}f_v^2 d\Sigma \\
&\leq& -n \int_{\Sigma}f_v^2 d\Sigma,\\
\end{eqnarray*}
where the last two inequalities come from our assumptions that $H =
\pm 1$ and $|A|^2 - 2nH^2 \geq 0$. If $f_v \not\equiv 0$ for all $v
\in \mathbb{R}^{n+2}$, then we have $Q(\psi_v) < 0$ for all $\psi_v
\in W$, as desired. Now, if $f_v \equiv 0$ for some $v$, then by
(\cite{NS68}, Theorem 1) we know that $\Sigma^n$ is totally
geodesic. This concludes the proof of the first part of the theorem.
To prove the second and third parts of the theorem, let us consider
the expansion
$$ Q(\psi_v) = -\int_{\Sigma}(|\phi|^2 -nH^2)l_v^2 d\Sigma - nH^2 \int_{\Sigma}f_v^2 d\Sigma +
nH(1-H^2)\int_{\Sigma}l_v f_v d\Sigma . $$

By our assumption that $|A|^2 \geq 2nH^2 $, we know that
$-\int_{\Sigma}(|\phi|^2 - nH^2)l_v^2 d\Sigma \leq 0$. Also, if
$\Sigma^n$ is not totally geodesic, $-nH^2 \int_{\Sigma}f_v^2
d\Sigma < 0$. We now want to estimate the third term in the
expansion. To do this, recall that
$$l_v \Delta l_v = -nl_v^2 + nHf_vl_v$$ and $$\frac{1}{2}
\Delta(l_v^2) = l_v \Delta l_v + |\nabla l_v|^2.$$
Then
$$ nH\int_{\Sigma}f_v l_v d\Sigma = -\int_{\Sigma}|\nabla l_v|^2
d\Sigma + n \int_{\Sigma}l_v^2 d\Sigma . $$

Using the hypotheses in the second and third parts, we find that

$$nH(1-H^2) \int_{\Sigma}f_v l_v d\Sigma \leq 0 . $$

Hence, under our assumptions, $Q(\psi_v) < 0$ for all $v$, so the
second and third parts are proved.
\end{proof}

\section{Acknowledgments}

The authors thank their adviser, Professor Abd\^enago Barros, for his invaluable help,
without which this paper would not have been possible. They also thank Professor Maria Helena Noronha,
who coordinated the 2008 International Research Experience for Students (IRES) during which the research
for this paper was done. The authors also gratefully acknowledge support from the
National Science Foundation (NSF-OISE Grant 0526008) and Brazil's
Conselho Nacional de Desenvolvimento Cient\'ifico e Tecnol\'ogico (CNPq Grant 451778/2008-1).


\begin{thebibliography}{1}

\bibitem{Alias06} Al\'ias, L.J. - On the stability index of minimal and constant mean curvature hypersurfaces
in spheres, Revista de la Uni\'on Matem\'atica Argentina, \textbf{47} (2006), 39--61.

\bibitem{ABP07} Al\'ias, L.J., Brasil, A. Jr. and Perdomo, O. - On the stability index of hypersurfaces with
constant mean curvature in spheres, Proc. Amer. Math. Soc., \textbf{135} (2007), 3685--3693.

\bibitem{BdCE88} Barbosa, J.L., do Carmo, M. and Eschenburg, J. - Stability of hypersurfaces with constant
mean curvature in Riemannian manifolds, Math. Z., \textbf{197} (1988), 123--138.

\bibitem{GBD99}Brasil, A. Jr., Delgado, J.A. and Guadalupe, I. - A characterization of the
Clifford torus, Rend. Circ. Mat. Palermo, \textbf{48} (1999), 537--540.



\bibitem{Soufi93} El Soufi, A. - Applications harmoniques, immersions minimales et transformations conformes de
la sph\'ere, Compositio Math., \textbf{85} (1993), 281--298.

\bibitem{Lawson80} Lawson, B. - Local rigidity theorems for minimal
hypersurfaces, Ann. of Math., \textbf{89} (1969), 187--197.


\bibitem{NS68}Nomizu, K. and Smyth, B. - On the $Gau\ss$ mapping for
hypersurfaces of constant curvature in the sphere, Comm. Math.
Helv., \textbf{44} (1969) , 484--490.

\bibitem{Perdomo01}Perdomo, O. - Low index minimal hypersurfaces of spheres,
Asian J. Math., \textbf{5} (2001), 741-750.

\bibitem{Simons68}Simons, J. - Minimal varieties in Riemannian
manifolds, Ann. of Math., \textbf{88} (1968), 62--105.

\bibitem{Urbano90}Urbano, F. - Minimal surfaces with low index in the
three-dimensional sphere, Proc. AMS, \textbf{108} (1990), 989--992.
\end{thebibliography}
\end{document}